\documentclass[letterpaper, 10 pt, conference]{ieeeconf}  

\IEEEoverridecommandlockouts                              
\usepackage{graphics} 
\usepackage{epsfig} 
\usepackage{times} 
\usepackage{amsmath} 
\usepackage{amssymb}  

\usepackage{multirow}
\usepackage{array}
\usepackage{amsfonts}

\usepackage{amsthm}

\usepackage{todonotes}
\usepackage{comment}
\usepackage{cite}

\usepackage{booktabs}

\theoremstyle{definition}
\newtheorem{assumption}{Assumption}
\newtheorem{definition}{Definition}
\theoremstyle{plain}

\newtheorem{corollary}{Corollary}
\newtheorem{lemma}{Lemma}

\newtheorem{remark}{Remark}

\newtheorem{problem}{Problem}

\title{\LARGE \bf
Optimal Path Planning for Connected and Automated Vehicles at Urban Intersections
}

\author{Andreas A. Malikopoulos, {\itshape{Senior Member, IEEE}}, and Liuhui Zhao, {\itshape{Member, IEEE}}
\thanks{This research was supported in part by ARPAE's NEXTCAR program under the award number DE-AR0000796 and by the  Delaware Energy Institute (DEI).}%
\thanks{The authors are with the Department of Mechanical Engineering, University of Delaware, Newark, DE 19716 USA (email: \tt\small{andreas@udel.edu}; \tt\small{lhzhao@udel.edu.}} }

\begin{document}

\maketitle
\thispagestyle{empty}
\pagestyle{empty}

\begin{abstract}
In earlier work, a decentralized optimal control framework was established for coordinating online connected and automated vehicles (CAVs) at urban intersections. The policy designating the sequence that each CAV crosses the intersection, however, was based on a first-in-first-out queue, imposing limitations on the optimal solution. Moreover, no lane changing, or left and right turns were considered. In this paper, we formulate an upper-level optimization problem, the solution of which yields, for each CAV, the optimal sequence and lane to cross the intersection. The effectiveness of the proposed approach is illustrated through simulation. 

\end{abstract}

\indent




\section{Introduction} \label{sec:1}
We are currently witnessing an increasing integration of our energy, transportation, and cyber networks, which, coupled with the human interactions, is giving rise to a new level of complexity in the transportation network. As we move to increasingly complex emerging mobility systems, new control approaches are needed to optimize the impact on system behavior of the interplay between vehicles at different transportation scenarios, e.g., intersections, merging roadways, roundabouts, speed reduction zones. These scenarios along with the driver responses to various disturbances \cite{Malikopoulos2013} are the primary sources of bottlenecks that contribute to traffic congestion \cite{Margiotta2011}. 

An automated transportation system \cite{Zhao2019} can alleviate congestion, reduce energy use and emissions, and improve safety by increasing significantly traffic flow as a result of closer packing of automatically controlled vehicles in platoons. One of the very early efforts in this direction was proposed in 1969 by Athans \cite{Athans1969} for safe and efficient coordination of merging maneuvers with the intention to avoid congestion.  Varaiya \cite{Varaiya1993} has discussed extensively the key features of an automated intelligent vehicle-highway system and proposed a related control system architecture. 


Connected and automated vehicles (CAVs) provide the most intriguing opportunity for enabling decision makers to better monitor transportation network conditions and make better operating decisions to improve safety and reduce pollution, energy consumption, and travel delays.  Several research efforts have been reported in the literature on coordinating CAVs at at different transportation scenarios, e.g., intersections, merging roadways, roundabouts, speed reduction zones. In 2004, Dresner and Stone \cite{Dresner2004} proposed the use of the reservation scheme to control a single intersection of two roads with vehicles traveling with similar speed on a single direction on each road. 
Since then, several  approaches have been proposed \cite{Dresner2008,DeLaFortelle2010} to maximize the throughput of signalized-free intersections including extensions of the reservation scheme in \cite{Dresner2004}. 
Some approaches have focused on coordinating vehicles at intersections to improve  travel time \cite{Yan2009}. Other approaches have considered minimizing the overlap in the position of vehicles inside the intersection, rather than arrival time \cite{Lee2012}. Kim and Kumar \cite{Kim2014} proposed an approach based on model predictive control that allows each vehicle to optimize its movement locally in a distributed manner with respect to any objective of interest.  
A detailed discussion of the research efforts in this area that have been reported in the literature to date can be found in \cite{Malikopoulos2016a}. 

In earlier work, a decentralized optimal control framework was established for coordinating online CAVs in different transportation scenarios, e.g., merging roadways, urban intersections, speed reduction zones, and roundabouts. The analytical solution without considering state and control constraints was presented in \cite{Rios-Torres2015}, \cite{Rios-Torres2}, and \cite{Ntousakis:2016aa} for coordinating online CAVs at highway on-ramps, in \cite{Zhang2016a} at two adjacent intersections, and in \cite{Malikopoulos2018a} at roundabouts. 
The solution of the unconstrained problem was also validated experimentally at the University of Delaware's Scaled Smart City using 10 CAV robotic cars \cite{Malikopoulos2018b} in a merging roadway scenario. The solution of the optimal control problem considering state and control constraints was presented in \cite{Malikopoulos2017} at an urban intersection. 

However, the policy designating the sequence that each CAV crosses the intersection in the aforementioned approaches, was based on a first-in-first-out queue, imposing limitations on the optimal solution. Moreover, no lane changing, or left and right turns were considered. 
In this paper, we formulate an upper-level optimization problem, the solution of which yields, for each CAV, the optimal sequence and lane to cross the intersection. The effectiveness of the solution is illustrated through simulation. 


The structure of the paper is organized as follows. In Section II, we formulate the problem of vehicle coordination at an urban intersection and provide the modeling framework. In Section III, we briefly present the analytical, closed form solution for the low-level optimization problem. In Section IV, we present the upper-level optimization problem the solution of yields, for each CAV, the optimal sequence and lane to cross the intersection. Finally in Section V, we validate the effectiveness of the solution through simulation. We offer concluding remarks in Section VI.

\section{Problem Formulation} \label{sec:2}

\subsection{Modeling Framework} \label{sec:2a}

We consider CAVs at a 100\% penetration rate crossing a signalized-free intersection (Fig. \ref{fig:1}). The region at the center of the intersection, called \textit{merging zone}, is the area of potential lateral collision of the vehicles. The intersection has a \textit{control zone} inside of which the CAVs can communicate with each other and with the intersection's \textit{crossing protocol}. The  \textit{crossing protocol}, defined formally in the next subsection, stores the vehicles' path trajectories from the time they enter until the time they exit the control zone. The distance from the entry of  the control zone until the entry of the merging zone is $S_c$ and, although it is not restrictive, we consider to be the same for all entry points of the control zone. We also consider the merging zone to be a square of side $S_m$ (Fig. \ref{fig:1}). Note that the length $S_c$ could be in the order of hundreds of $m$ depending on the crossing protocol's communication range capability, while $S_m$ is the length of a typical intersection. The CAVs crossing the intersection can also make a right turn of radius $R_r$, or a left turn of radius $R_l$ (Fig. \ref{fig:1}). The intersection's geometry is not restrictive in our modeling framework, and is used only to determine the total distance travelled by each CAV inside the control zone. 

\begin{figure}
	\centering
	\includegraphics[width=3.4 in]{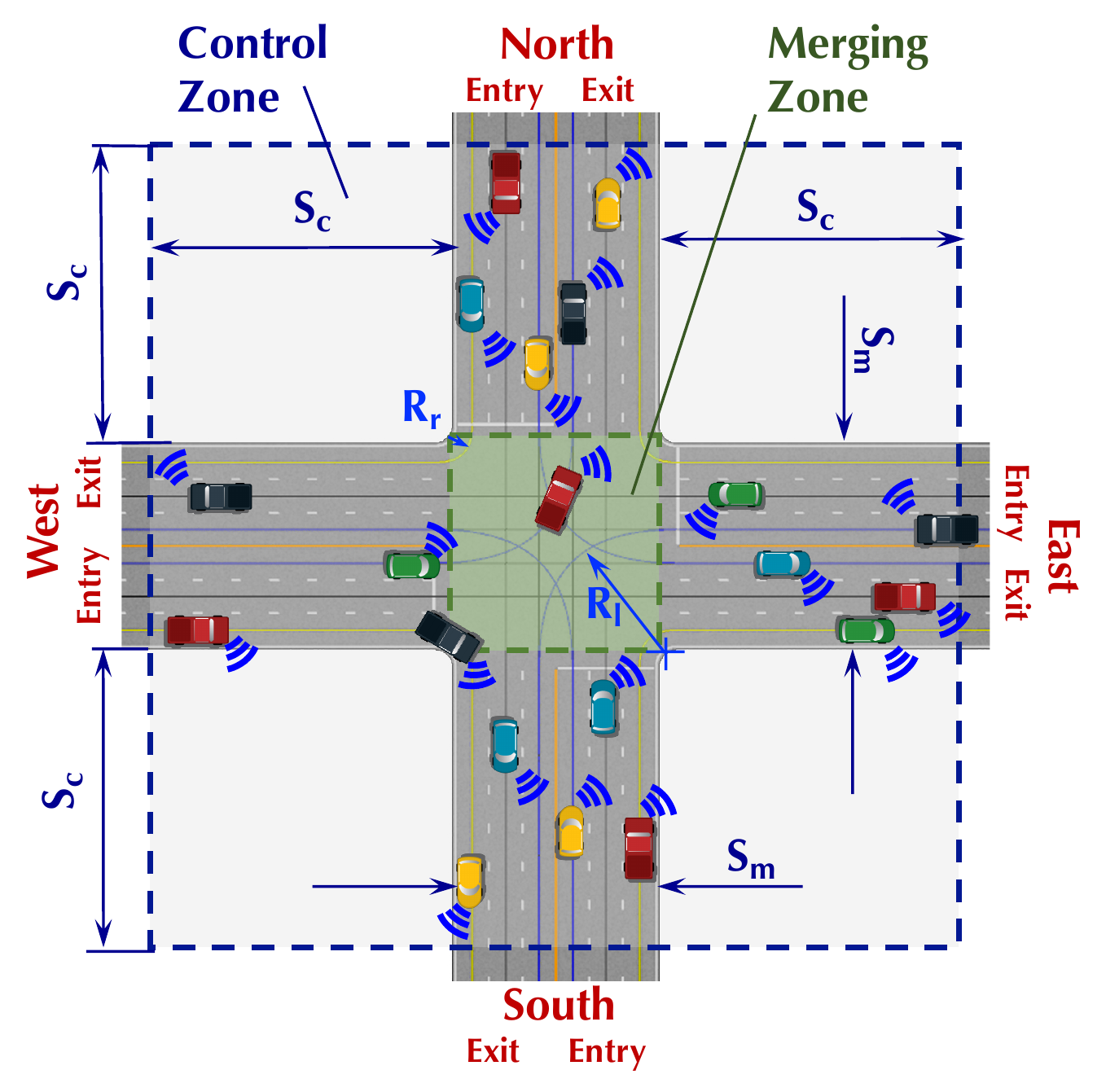} 
	\caption{A signalized-free intersection.}%
	\label{fig:1}%
\end{figure}

Let $\mathcal{N}(t)=\{1,\ldots,N(t)\}$, $N(t)\in\mathbb{N}$, be the set of CAVs inside the control zone at time $t\in\mathbb{R}^{+}$. Let $t_{i}^{f}$ be the assigned time for vehicle $i$ to exit the control zone.
There is a number of ways to assign $t_{i}^{f}$ for each vehicle $i$. For example, we may
impose a strict first-in-first-out queuing structure \cite{Malikopoulos2017}, where each CAV must
exit the control zone in the same order it entered the control zone. The policy, which  determines the time $t_{i}^{f}$ that each vehicle $i$ exits the control zone,  is the
result of an upper-level optimization problem and can aim at maximizing the throughput of the intersection. On the other hand, deriving the optimal control input (minimum acceleration/deceleration) for each vehicle $i$ from the time $t_{i}^{0}$ it enters the control zone to achieve the target $t_{i}^{f}$ can aim at minimizing its energy \cite{Malikopoulos2010a}.

In what follows, we present a two-level, joint optimization framework: (1) an upper level optimization that yields for each CAV $i\in\mathcal{N}(t)$ with a given origin (entry of the control zone) and desired destination (exit of the control zone)  the sequence that will be exiting the control zone, namely, (a) minimum time $t_{i}^{f}$ to exit the control zone and (b) optimal path including the lanes that each CAV should be occupying while traveling inside the control zone; and (2) a low-level optimization that yields, for CAV $i\in\mathcal{N}(t),$ its optimal control input (acceleration/deceleration) to achieve the optimal path and $t_{i}^{f}$ derived in (1) subject to the state, control, and safety constraints.

The two-level optimization framework is used by each CAV $i\in\mathcal{N}(t)$ as follows. When vehicle $i$ enters the control zone at $t_{i}^{0}$, it accesses the intersection's \textit{crossing protocol} that includes the path trajectories, defined formally in the next subsection, of all CAVs inside the control zone. Then, vehicle $i$ solves the upper-level optimization problem and derives the minimum time $t_{i}^{f}$ to exit the control zone along with its optimal path including the appropriate lanes that it should occupy. The outcome of the upper-level optimization problem becomes the input of the low-level optimization problem. In particular, once the CAV derives the minimum time $t_{i}^{f}$, it derives its minimum acceleration/deceleration profile, in terms of energy, to achieve the exit time $t_{i}^{f}$. 

The implications of the proposed optimization framework are that CAVs do not have to come to a full stop at the intersection, thereby conserving momentum and energy while also improving travel time. Moreover, by optimizing each vehicle's acceleration/deceleration, we minimize transient engine operation  \cite{Malikopoulos2008b}, and thus we have additional benefits in fuel consumption.

\subsection{Vehicle Model, Constraints, and Assumptions} \label{sec:2b}

In our analysis, we consider that 
each CAV $i\in\mathcal{N}(t)$ is governed by the following  dynamics
\begin{equation}%
\begin{split}
\dot{p}_{i} &  =v_{i}(t)\\
\dot{v}_{i} &  =u_{i}(t)\\
\dot{s}_{i} &  = \xi_i \cdot (v_{k}(t)-v_{i}(t))
\label{eq:model2}
\end{split}
\end{equation}
where $p_{i}(t)\in\mathcal{P}_{i}$, $v_{i}(t)\in\mathcal{V}_{i}$, and
$u_{i}(t)\in\mathcal{U}_{i}$ denote the position, speed and
acceleration/deceleration (control input) of each vehicle $i$ inside the control zone at time $t\in[t_{i}^{0}, t_{i}^{f}]$, where $t_i^0$ and $t_i^f$ are the times that vehicle $i$ enters and exits the control zone respectively; ~$s_{i}(t)\in\mathcal{S}_{i}$, with $s_{i}(t)=p_{k}(t)-p_{i}(t),$ denotes the distance of vehicle $i$ from the CAV $k\in\mathcal{N}(t)$ which is physically immediately ahead of $i$ in the same lane, and $\xi_{i}$ is a reaction constant of vehicle $i$. The sets $\mathcal{P}_{i}$,$\mathcal{V}_{i}$, $\mathcal{U}_{i}$, and $\mathcal{S}_{i}$, $i\in\mathcal{N}(t),$ are complete and totally bounded subsets of $\mathbb{R}$.

Let $x_{i}(t)=\left[p_{i}(t) ~ v_{i}(t) ~ s_{i}(t)\right]  ^{T}$ denote the state of each vehicle $i$ taking values in $\mathcal{X}_{i}%
=\mathcal{P}_{i}\times\mathcal{V}_{i}\times\mathcal{S}_{i}$, with initial value
$x_{i}(t_{i}^{0})=x_{i}^{0}=\left[p_{i}^{0} ~ v_{i}^{0} ~s_{i}^{0}\right]  ^{T},$ where $p_{i}^{0}= p_{i}(t_{i}^{0})=0$, $v_{i}^{0}= v_{i}(t_{i}^{0})$, and $s_{i}^{0}= s_{i}(t_{i}^{0})$ at the entry of the control zone.  The state space 
$\mathcal{X}_{i}$ for each vehicle $i$ is
closed with respect to the induced topology on $\mathcal{P}_{i}\times
\mathcal{V}_{i}\times\mathcal{S}_{i}$ and thus, it is compact.
We need to ensure that for any initial state $(t_i^0, x_i^0)$ and every admissible control $u(t)$, the system \eqref{eq:model2} has a unique solution $x(t)$ on some interval $[t_i^0, t_i^f]$. 
The following observations from \eqref{eq:model2} satisfy some regularity conditions required both on the state equations and admissible controls $u(t)$ to guarantee local existence and uniqueness of solutions for \eqref{eq:model2}: a) the state equations are continuous in $u$ and continuously differentiable in the state $x$, b) the first derivative of the state equations in $x$, is continuous in $u$, and c) the admissible control $u(t)$ is continuous with respect to $t$.

To ensure that the control input and vehicle speed are within a
given admissible range, the following constraints are imposed.
\begin{gather}%
u_{i,min}  \leq u_{i}(t)\leq u_{i,max}, \label{speed_accel constraints}  \quad\text{and}\\
0 < v_{min}\leq v_{i}(t)\leq v_{max},\label{speed}\quad\forall t\in\lbrack t_{i}%
^{0},t_{i}^{f}],
\end{gather}
where $u_{i,min}$, $u_{i,max}$ are the minimum deceleration and maximum
acceleration for each vehicle $i\in\mathcal{N}(t)$, and $v_{min}$, $v_{max}$ are the minimum and maximum speed limits respectively. 

To ensure the absence of rear-end collision of two consecutive vehicles traveling on the same lane,  the position of the preceding vehicle should be greater than or equal to the position of the following vehicle plus a predefined safe distance $\delta_i(t)$. Thus we impose the rear-end safety constraint 
\begin{equation}
\begin{split}
s_{i}(t)=\xi_i \cdot (p_{k}(t)-p_{i}(t)) \ge \delta_i(t),~ \forall t\in [t_i^0, t_i^f].
\label{eq:rearend}
\end{split}
\end{equation}
We consider  constant time headway instead of constant distance that each vehicle should keep when following the other vehicles, thus, the minimum safe distance $\delta_i(t)$ is expressed as a function of speed $v_i(t)$ and minimum time headway between vehicle $i$ and its preceding vehicle $k$, denoted as $\rho_i$.
\begin{equation}
\begin{split}
\delta_i(t)=\gamma_i + \rho_i \cdot v_i(t),~ \forall t\in [t_i^0, t_i^f],
\label{eq:safedist}
\end{split}
\end{equation}
where $\gamma_i$ is the standstill distance (i.e., the distance between two vehicles when they both stop).

A lateral collision can occur if a vehicle $j\in\mathcal{N}(t)$ cruising on a different road from $i$ inside the merging zone. In this case, the lateral safety constraint between $i$ and $j$ is 
\begin{equation}
\begin{split}
s_{i}(t)=\xi_i \cdot (p_{j,i}(t)-p_{i}(t)) \ge \delta_i(t),~ \forall t\in [t_i^0, t_i^f],
\label{eq:lateral}
\end{split}
\end{equation}
where $p_{j,i}(t)$ is the distance of vehicle $j$ from the entry point that vehicle $i$ entered the control zone.

\begin{definition}	\label{def:lanes}
	The set of all lanes at the roads of the intersection is denoted by $\mathcal{L}:=\{1,\dots,M\}, M\in\mathbb{N}.$	
\end{definition}

\begin{definition} \label{def:lanesfunction}
	For each vehicle $i\in\mathcal{N}(t)$, the function $l_i(t): [t_i^0, t_i^f]\to \mathcal{L}$ yields the lane the vehicle $i$ occupies inside the control zone at time $t$.
\end{definition}

\begin{definition} \label{def:cardinal}
	For each vehicle $i\in\mathcal{N}(t)$, the pair of the cardinal point that the vehicle enters the control zone and the cardinal point that the vehicle exits the control zone is denoted by $o_i$.
	
\end{definition}

For example, based on Definition \ref{def:cardinal}, for a vehicle $i$ that enters the control zone from the West entry (Fig. \ref{fig:1}) and exits the control zone from the South exit, $o_i=(W,S)$.

\begin{definition}\label{def:path}
	For each vehicle $i\in\mathcal{N}(t)$, the function $t_{p_i,l_i}\big(p_i(t),l_i(t)\big): \mathcal{P}_i\times \mathcal{L}\to[t_i^0, t_i^f],$ is called the \textit{path trajectory} of vehicle $i$, and it yields the time when vehicle $i$ is at the position $p_i(t)$ inside the control zone and occupies lane $l_i(t)$.
	
\end{definition}

\begin{definition}\label{def:protocol}
	The intersection's \textit{crossing protocol} is denoted by $\Pi(t)$ and includes the following information
	\begin{gather}\label{eq:protocol}
		\Pi(t):=\{t_{p_i,l_i}\big(p_i(t),l_i(t)\big), l_i(t), o_i, t_i^0, t_i^f\},  \\ \nonumber
		\forall i\in\mathcal{N}(t), t\in\mathbb{R}^+.
	\end{gather}
\end{definition}


\begin{remark} \label{ass:feas}
	 The vehicles traveling inside the control zone can change lanes either (1) in the lateral direction (e.g., move to a neighbor lane), or (2) when making a right (or a left) turn inside the merging zone. In the former case, when the vehicle changes lane it travels along the hypotenuse $dy$ of the triangle created by the width of the lane and the longitudinal displacement $dp$ if it had not changed lane. Thus, in this case, the vehicle travels an additional distance which is equal to the difference between the hypotenuse $dy$ and the longitudinal displacement $dp$, i.e., $dy-dp$.
\end{remark}


\begin{remark} \label{ass:feas}
	 When a vehicle is about to make a right turn it must occupy the right lane of the road before it enters the merging zone. Similarly, when a vehicle is about to make a left turn it must occupy the left lane before it enters the merging zone.
\end{remark}

In the modeling framework presented above, we impose the following assumptions:


\begin{assumption} \label{ass:lane} 
	The vehicle's additional distance $dy-dp$ traveled when it changes lanes in the lateral direction can be neglected. 
\end{assumption}


\begin{assumption} \label{ass:noise}
	 Each CAV $i\in\mathcal{N}(t)$ has proximity sensors and can communicate with other CAVs and the \textit{crossing protocol} without any errors or delays.
\end{assumption}

The first assumption can be justified since we consider an intersection and the speed limit inside the control zone is relatively low, hence $dy\approx dp$. The second assumption may be strong, but it is relatively straightforward to relax it as long as the noise in the communication, measurements and delays are bounded. In this case, we can determine upper bounds on the state uncertainties as a result of sensing or communication errors and delays, and incorporate these into more conservative safety constraints.

When each vehicle $i$ with a given $o_i$ enters the control zone, it accesses the intersection's \textit{crossing protocol} and solves two optimization problems: (1) an upper-level optimization problem, the solution of which yields its path trajectory $t_{p_i,l_i}\big(p_i(t),l_i(t)\big)$ and the minimum time $t_{i}^{f}$ to exit the control zone; and (2) a low-level optimization problem, the solution of which yields its optimal control input (acceleration/deceleration) to achieve the optimal path and $t_{i}^{f}$ derived in (1) subject to the state, control, and safety constraints.

We start our exposition with the low-level optimization problem, and then we discuss the upper-level problem.

\section{Low-level optimization} \label{sec:3}

In this section, we consider that the solution of the upper-level optimization problem is given, and thus, the minimum time $t_{i}^{f}$ for each vehicle $i\in\mathcal{N}(t)$ is known, and we focus on a low-level optimization problem that yields for each vehicle $i$ the optimal control input (acceleration/deceleration) to achieve the assigned $t_{i}^{f}$ subject to the state, control, and safety constraints.

\begin{problem} \label{problem1}
Once $t_{i}^{f}$ is determined, the low-level problem for each vehicle $i\in\mathcal{N}(t)$ is to minimize the cost functional $J_{i}(u(t))$, which is the $L^2$-norm of the control input in $[t_i^0, t_i^f]$
\begin{gather}\label{eq:decentral}
\min_{u(t)\in U_i} J_{i}(u(t))=  \frac{1}{2} \int_{t^0_i}^{t^f_i} u^2_i(t)~dt,\\ 
\text{subject to}%
:\eqref{eq:model2},\eqref{speed_accel constraints},\eqref{speed}, \eqref{eq:rearend},\nonumber\\
\text{and given }t_{i}^{0}\text{, }v_{i}^{0}\text{, }t_{i}^{f}\text{,
}p_{i}(t_{i}^{0})\text{, }p_{i}(t_{i}^{f}),\nonumber
\end{gather}
where $p_{i}(t_{i}^{0})=0$, while  the value of $p_{i}(t_{i}^{f})$ for each $i\in\mathcal{N}(t)$ depends on $o_i$ and, based on Assumption \ref{ass:lane}, can take the following values (Fig. \ref{fig:1}): (1) $p_{i}(t_{i}^{f})=2 S_c + S_m$, if the CAV crosses the merging zone, (2) $p_{i}(t_{i}^{f})=2 S_c + \frac{\pi R_r}{2}$, if the CAV makes a right turn at the merging zone, and (3) $p_{i}(t_{i}^{f})=2 S_c + \frac{\pi R_l}{2}$, if the CAV makes a left turn at the merging zone.
\end{problem}
For the analytical solution of \eqref{eq:decentral}, we formulate the Hamiltonian 
\begin{gather}
H_{i}\big(t, p_{i}(t), v_{i}(t), s_{i}(t), u_{i}(t)\big)  \nonumber \\
=\frac{1}{2} u_i(t)^{2}_{i} + \lambda^{p}_{i} \cdot v_{i}(t) + \lambda^{v}_{i} \cdot u_{i}(t) +\lambda^{s}_{i} \cdot \xi_i \cdot (v_{k}(t) - v_{i}(t)) \nonumber\\
+ \mu^{a}_{i} \cdot(u_{i}(t) - u_{max})
+ \mu^{b}_{i} \cdot(u_{min} - u_{i}(t)) \nonumber\\
+ \mu^{c}_{i} \cdot  u_{i}(t) - \mu^{d}_{i} \cdot u_{i}(t) \nonumber\\ 
+ \mu^{s}_{i} \cdot (\rho_i \cdot u_i(t) - \xi_i\big(v_{k}(t) - v_i(t)\big)) ,\label{eq:16b}
\end{gather}
where $\lambda^{p}_{i}$, $\lambda^{v}_{i}$, and $\lambda^{s}_{i}$ are the influence functions \cite{Bryson:1963}, and
$\mu^{T}$ is the vector of the Lagrange multipliers. To address this problem, the constrained and unconstrained arcs will be pieced together to satisfy the Euler-Lagrange equations and necessary condition of optimality. 

For the case that none of the state and control constraints become active, the optimal control is \cite{Malikopoulos2019ACC}
\begin{equation}
u^{*}_{i}(t) = (a_{i} - b_{i} \cdot \xi_i) \cdot t + c_{i}, ~ t \in[t^{0}_{i}, t_i^f]. \label{eq:20}
\end{equation}
Substituting the last equation into \eqref{eq:model2} we find the optimal speed and position for each vehicle,
namely
\begin{gather}
v^{*}_{i}(t) = \frac{1}{2} (a_{i} - b_{i} \cdot \xi_i) \cdot t^2 + c_{i} \cdot t +d_{i}, ~ t \in[t^{0}_{i}, t_i^f], \label{eq:21}\\
p^{*}_{i}(t) = \frac{1}{6} (a_{i} - b_{i} \cdot \xi_i) \cdot t^3 +\frac{1}{2} c_{i} \cdot t^2 + d_{i}\cdot t +e_{i}, \label{eq:22} \\~ t \in[t^{0}_{i}, t_i^f],  \nonumber 
\end{gather}
where $a_{i}$, $b_{i}$, $c_{i}$, $d_{i}$ and $e_{i}$ are constants of integration that can be computed by the initial, final, and transversality conditions \cite{Malikopoulos2019ACC}.

\section{Upper-level optimization} \label{sec:4}

When a vehicle $i\in\mathcal{N}(t),$ with a given $o_i$, enters the control zone, it accesses the intersection's \textit{crossing protocol} and solves an upper-level optimization problem. The solution of this problem yields for $i$ the path trajectory $t_{p_i,l_i}\big(p_i(t),l_i(t)\big)$ and the minimum time $t_{i}^{f}$ to exit the control zone. 
In our exposition, we seek to derive the minimum $t_{i}^{f}$ without activating any of the state and control constraints of the low-level optimization Problem \ref{problem1}. Therefore, the upper-level optimization problem  should yield a $t_{i}^{f}$ such that the solution of the low-level optimization problem will result in the unconstrained case \eqref{eq:20} - \eqref{eq:22}. 

There is an apparent trade off between the two problems. The lower the value of $t_{i}^{f}$ in the upper-level problem, the higher the value of the control input in $[t_{i}^{0}, t_{i}^{f}]$ in the low-level problem. 
The low-level problem is directly related to minimizing energy for each vehicle (individually optimal solution). On the other hand, the upper-level problem is related to maximizing the throughput of the intersection, thus eliminating stop-and-go driving (social optimal solution). Therefore, by seeking a solution for the upper-level problem which guarantees  that none of the state and control constraints become active may be considered an appropriate compromise between the two. 

For simplicity of notation, for each vehicle $i\in\mathcal{N}(t)$ we write the optimal position \eqref{eq:22} of the unconstrained case in the following form

\begin{gather}
p^{*}_{i}(t) = \phi_{i,3} \cdot t^3 +\phi_{i,2} \cdot t^2 + \phi_{i,1} \cdot t +\phi_{i,0} , ~ t\in [t_{i}^{0}, t_{i}^{f}], \label{eq:upper_p}%
\end{gather}
where $\phi_{i,3}, \phi_{i,2}, \phi_{i,1}, \phi_{i,0}\in\mathbb{R}$ are the constants of integration derived in the Hamiltonian analysis, in Section \ref{sec:3}, for the unconstrained case.

\begin{remark} \label{rem:3}
	For each $i\in\mathcal{N}(t),$ the optimal position \eqref{eq:upper_p} is a continuous and differentiable function. Based on \eqref{speed}, it is also an increasing function with respect to $t\in\mathbb{R}^+$.
\end{remark}

%

Next, we investigate some properties of \eqref{eq:upper_p}.

\begin{lemma} \label{lem:1}
For each $i\in\mathcal{N}(t)$, the optimal position $p_i^*$ given by \eqref{eq:upper_p}  is an one-one function.
\end{lemma}
\begin{proof}
	Since, for each $i\in\mathcal{N}(t),$ $p_i^*(t)$ is an increasing function with respect to $t\in\mathbb{R}^+$ and from \eqref{speed}, for any $t_1, t_2\in[t_{i}^{0}, t_{i}^{f}]$, $p_i^*(t_{1})\neq p_i^*(t_{2}).$
	\end{proof}

\begin{corollary} \label{cor:1}
	Since, for each $i\in\mathcal{N}(t),$  \eqref{eq:upper_p} is an one-one function, there exist an inverse function $p_i^*(t)^{-1}$ such that
\begin{gather} \label{eq:upper_inversep}
p^{*}_{i}(t)^{-1} = \omega_{i,3} \cdot p^3 +\omega_{i,2} \cdot p^2 + \omega_{i,1} \cdot p +\omega_{i,0} , 
\end{gather}	
where $\omega_{i,3}, \omega_{i,2}, \omega_{i,1}, \omega_{i,0}\in\mathbb{R}$ are constants that are a function of $\phi_{i,3}, \phi_{i,2}, \phi_{i,1}, \phi_{i,0}$.
\end{corollary}

\begin{remark} \label{rem:4}
	For each $i\in\mathcal{N}(t),$  $t\in [t_{i}^{0}, t_{i}^{f}]$,  we rewrite \eqref{eq:upper_p} as follows
	\begin{gather}
	p^{*}_{i}(t) = \phi_{i,3} \cdot t_i^3 +\phi_{i,2} \cdot t_i^2 + \phi_{i,1} \cdot t_i +\phi_{i,0}. \label{eq:upper_pi}%
	\end{gather}
\end{remark}

\begin{lemma} \label{lem:2}
	Let $p^{*}_{i}(t)^{-1}$ be the inverse function of \eqref{eq:upper_p} for each vehicle $i\in\mathcal{N}(t).$ Then the constants $\phi_{i,3}, \phi_{i,2}, \phi_{i,1}, \phi_{i,0}\in\mathbb{R}$ can be derived by $\omega_{i,3}, \omega_{i,2}, \omega_{i,1}, \omega_{i,0}\in\mathbb{R}.$ 
\end{lemma} 
\begin{proof}
	Due to space limitation the proof is omitted. However, the result is trivial. 
\end{proof}

\begin{remark} \label{rem:5}
	The inverse function $p_i^*(t)^{-1}=t_i(p^*(t)),$ where $t_i(p^*(t))\in[t_{i}^{0}, t_{i}^{f}]$, yields the time that vehicle $i\in\mathcal{N}(t)$ is at the position $p^{*}_{i}(t)$ inside the control zone. 
\end{remark}

\begin{lemma} \label{lem3}
	For each $i\in\mathcal{N}(t),$ the domain of $t_i(p^*(t))$ is the closed interval $[p_i(t_{i}^{0}),p_i(t_{i}^{f})]$.
\end{lemma} 

\begin{proof}
	Since, for each $i\in\mathcal{N}(t),$ $p_i^*(t)$ is an increasing function in $[t_{i}^{0}, t_{i}^{f}]$, then by the Intermediate Value Theorem, $p_i^*(t)$ takes values on the closed interval $[p_i(t_{i}^{0}),p_i(t_{i}^{f})]$. 
\end{proof}

\begin{corollary} \label{cor:2}
	Since $p_i^*(t)$ is a continuous and one-one function in $[t_{i}^{0}, t_{i}^{f}]$ for each $i\in\mathcal{N}(t),$ then $t_i(p^*(t))$ is also continuous.
\end{corollary}

\begin{corollary} \label{cor:3}
	For each $i\in\mathcal{N}(t),$ $p'\big(t_i(p(t)) \big)\neq 0$ for all $p\in [p_i(t_{i}^{0}),p_i(t_{i}^{f})]$. Hence, $t_i(p^*(t))$ is differentiable in $[p_i(t_{i}^{0}),p_i(t_{i}^{f})]$.
\end{corollary}

\begin{lemma} \label{lem4}
	For each $i\in\mathcal{N}(t),$ $t_i(p^*(t))$ is an increasing function in $[p_i(t_{i}^{0}),p_i(t_{i}^{f})]$.	
\end{lemma} 

\begin{proof}
	From Lemma \ref{lem3}, for each $i\in\mathcal{N}(t)$ the domain of $t_i(p^*(t))$ is $[p_i(t_{i}^{0}), p_i(t_{i}^{f})]$. Let $p_i(t_{i}^{0})<\alpha_1<\alpha_2< p_i(t_{i}^{f})$ with  $t_i(p_i(t_{i}^{0}) < t_i(\alpha_1)$. If we had $t_i(p_i(t_{i}^{0})) > t_i(\alpha_2),$ then by applying the Intermediate Value Theorem to the interval  $[\alpha_2, p_i(t_{i}^{f})]$ would give an $\alpha_3$ with $\alpha_2<\alpha_3< p_i(t_{i}^{f})$ and $t_i(p_i(t_{i}^{0})) = t_i(\alpha_3)$ contradicting the fact that $t_i(p^*(t))$ is one-one on $[p_i(t_{i}^{0}),p_i(t_{i}^{f})]$.
\end{proof}

Since each vehicle $i\in\mathcal{N}(t)$ can change lanes inside the control zone, its position should be associated with the function $l_i(t)$ (Definition \ref{def:lanesfunction}) that yields the lane vehicle $i$ occupies inside the control zone at $t$.

\begin{definition}	\label{def:pos_lane}
	The position of each vehicle $i\in\mathcal{N}(t)$ using lane $l_i(t)$, $m\in\mathcal{L},$ is denoted by $p_{i,l}(t,l)$.
\end{definition}

Based on Definition \ref{def:pos_lane}, we  augment the optimal position of $i\in\mathcal{N}(t)$  given by \eqref {eq:upper_p} to capture the lane that vehicle $i$ as follows
\begin{gather}
p^*_{i}(t,l) = p^*(t)\cdot {I}_{1}(l) +  p^*(t)\cdot {I}_{2}(l) + \dots + p^*(t)\cdot {I}_{M}(l), \nonumber \\
t\in [t_{i}^{0}, t_{i}^{f}], \label{eq:upper_pl}%
\end{gather}
where  ${I}_{m}(l)$, $m\in\mathcal{L}$,  is the indicator function with ${I}_{m}(l=m)=1$, if $i$ occupies lane $m\in\mathcal{L}$ and ${I}_{m}(l\neq m)=0$ otherwise. 
For each vehicle $i\in\mathcal{N}(t)$, the inverse function of \eqref {eq:upper_p} enhanced with the lane that vehicle $i$ occupies is the path trajectory (Definition \ref{def:path})  and can be written as follows
\begin{gather} 
t_{p_i,l_i}(p_i(t),l_i(t))) = \omega_{i,3} \cdot p^3_{i}(t,l) +\omega_{i,2} \cdot p^2_{i}(t,l) \nonumber \\
+\omega_{i,1} \cdot p_{i}(t,l) +\omega_{i,0} . \label{eq:path_traj}
\end{gather}	

The path trajectory $t_{p_i,l_i}(p_i(t),l_i(t))) $ yields the time that vehicle $i$ is at the position $p_i(t)$ inside the control zone and occupies lane $l_i(t)$ and is used as the cost function for the upper-level optimization problem. 

In the upper-level optimization problem, each vehicle $i\in\mathcal{N}(t)$ derives its optimal path trajectory which yields the minimum time $t_i^f$ that vehicle $i$ exits the control zone along with the lane $l^*\in\mathcal{L}$ that should occupy at each $p_i^*$. To formulate this problem, we need to minimize \eqref{eq:path_traj}, evaluated at $p_i(t_i^f),$ with respect to $\omega_{i,3}, \omega_{i,2}, \omega_{i,1}, \omega_{i,0}$ that determine the shape of the path trajectory of the vehicle in $[p_i^0, p_i^f]$. Note that the value of $p_{i}(t_{i}^{f})$ for each $i\in\mathcal{N}(t)$ depends on $o_i$ and, based on the Assumption \ref{ass:lane}, it can be equal to (see Fig. \ref{fig:1}): (1) $p_{i}(t_{i}^{f})=2 S_c + S_m$, if the vehicle crosses the merging zone, (2) $p_{i}(t_{i}^{f})=2 S_c + \frac{\pi R_r}{2}$, if the vehicle makes a right turn at the merging zone, and (3) $p_{i}(t_{i}^{f})=2 S_c + \frac{\pi R_l}{2}$, if the vehicle makes a left turn at the merging zone. For simplicity of notation, we denote the total distance travelled by the vehicle $i\in\mathcal{N}(t)$ in $[t_{i}^{0}, t_{i}^{f}]$ with $S_{i,total}$, thus $p_{i}(t_{i}^{f})=S_{i,total}$.
Hence, the upper-level optimization problem is formulated as follows.

\begin{problem} \label{problem2}

	\begin{gather}\label{eq:decentral2}
	\min_{\omega_{i,3}, \omega_{i,2}, \omega_{i,1}, \omega_{i,0}}t_{p_i,l_i}\big(S_{i,total},l_i(t)\big)\\
	\text{subject to}%
	: \eqref{speed_accel constraints},\eqref{speed}, \eqref{eq:rearend}, \text{and given }t_{i}^{0}\text{, }v_{i}^{0}\text{, }t_{i}^{f}\text{,
	}p_{i}(t_{i}^{0})\text{, }p_{i}(t_{i}^{f}).\nonumber
	\end{gather}
\end{problem}
From Lemma \ref{lem:2}, the constants $\phi_{i,3}, \phi_{i,2}, \phi_{i,1}, \phi_{i,0}\in\mathbb{R}$ corresponding to the constraints imposed through \eqref{eq:upper_p} can be derived by $\omega_{i,3}, \omega_{i,2}, \omega_{i,1}, \omega_{i,0}\in\mathbb{R}$. This is a nonlinear programming problem  that each vehicle can solve using Lagrange multiplier theory.

\section{Simulation Results} \label{sec:5}
\subsection{Validation of Upper-Level Optimization} \label{sec:5a}

To evaluate the effectiveness of the solution of the proposed upper-level optimization problem, we conduct a simulation in MATLAB. The simulation setting is as follows. The intersection contains two roads, each of which has one lane per direction. The length of each direction is 300 $m$, the merging zone of the intersection is 25 $m$ by 25 $m$, and the entry of merging zone is located at 125 $m$ from the entry point for both directions. The maximum and minimum speed are 18 $m/s$ and 2 $m/s$, respectively. The maximum and minimum acceleration are 3.0 $m/s^2$ and -3.0 $m/s^2$. The safety (minimum allowed) headway is 1.0 $s$, and the standstill distance is 1.5 $m$. Six vehicles are entering into the intersection from three directions at different time steps.

Since vehicle 1 is the first vehicle in the network, it cruises through the intersection without any constraints imposed. The trajectories of all  vehicles along with the safety distance are shown in Fig. ~\ref{fig:distance}. Negative values of the safety distance means violation of the rear-end constraint. We see  from Fig.~\ref{fig:distance} that both rear-end and lateral collision constraints are satisfied. The control input (acceleration) and speed profiles for the vehicles in the network is shown in Fig.~\ref{fig:speed}. We note that for all  vehicles driving through the intersection, none of the acceleration and speed constraints are activated.

\begin{figure}[!h]
    \centering
    \includegraphics[width=.48\textwidth]{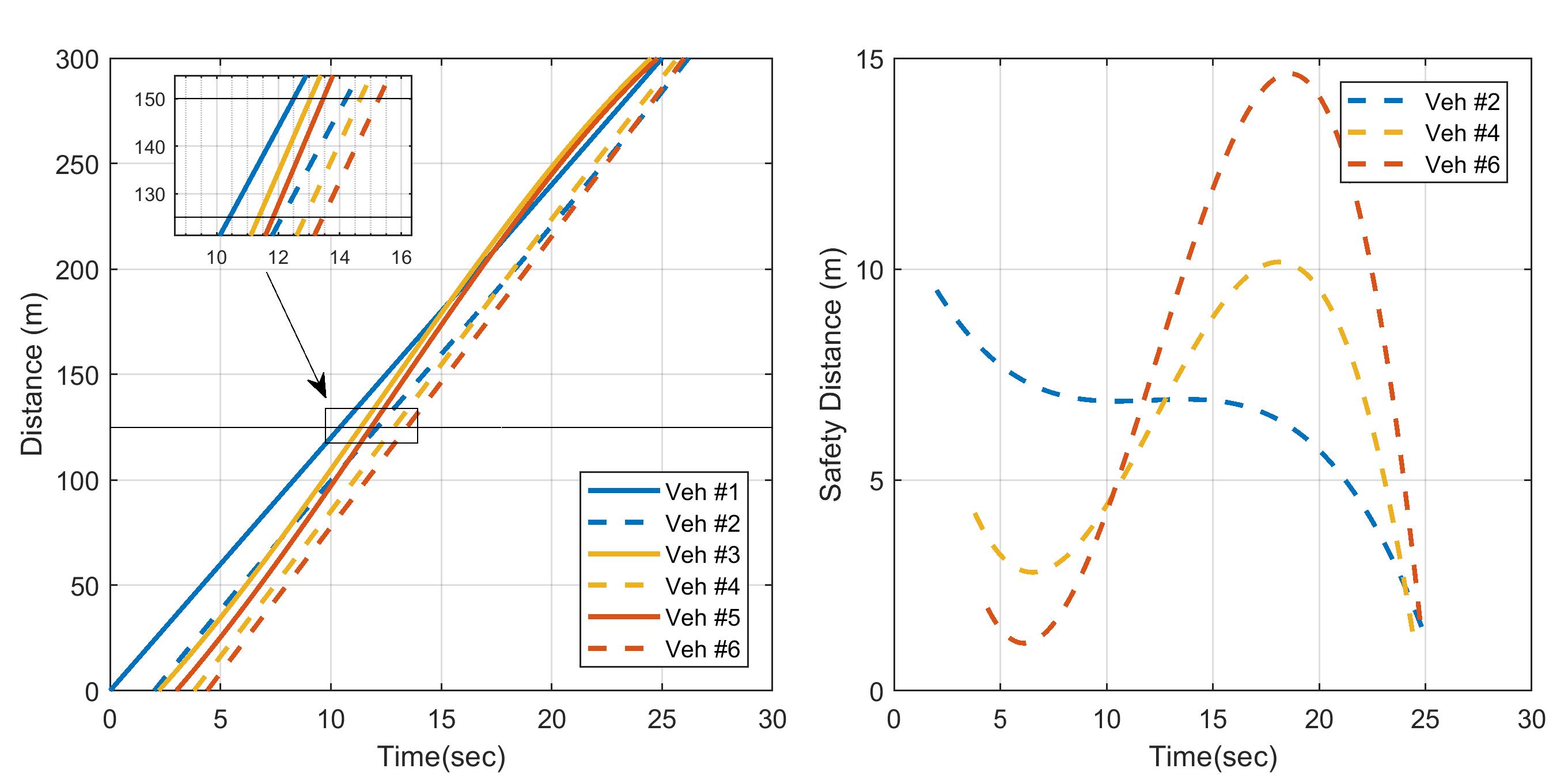}
    \caption{Trajectories and safety distances of vehicles.}
    \label{fig:distance}
\end{figure}

\begin{figure}[!h]
    \centering
    \includegraphics[width=.48\textwidth]{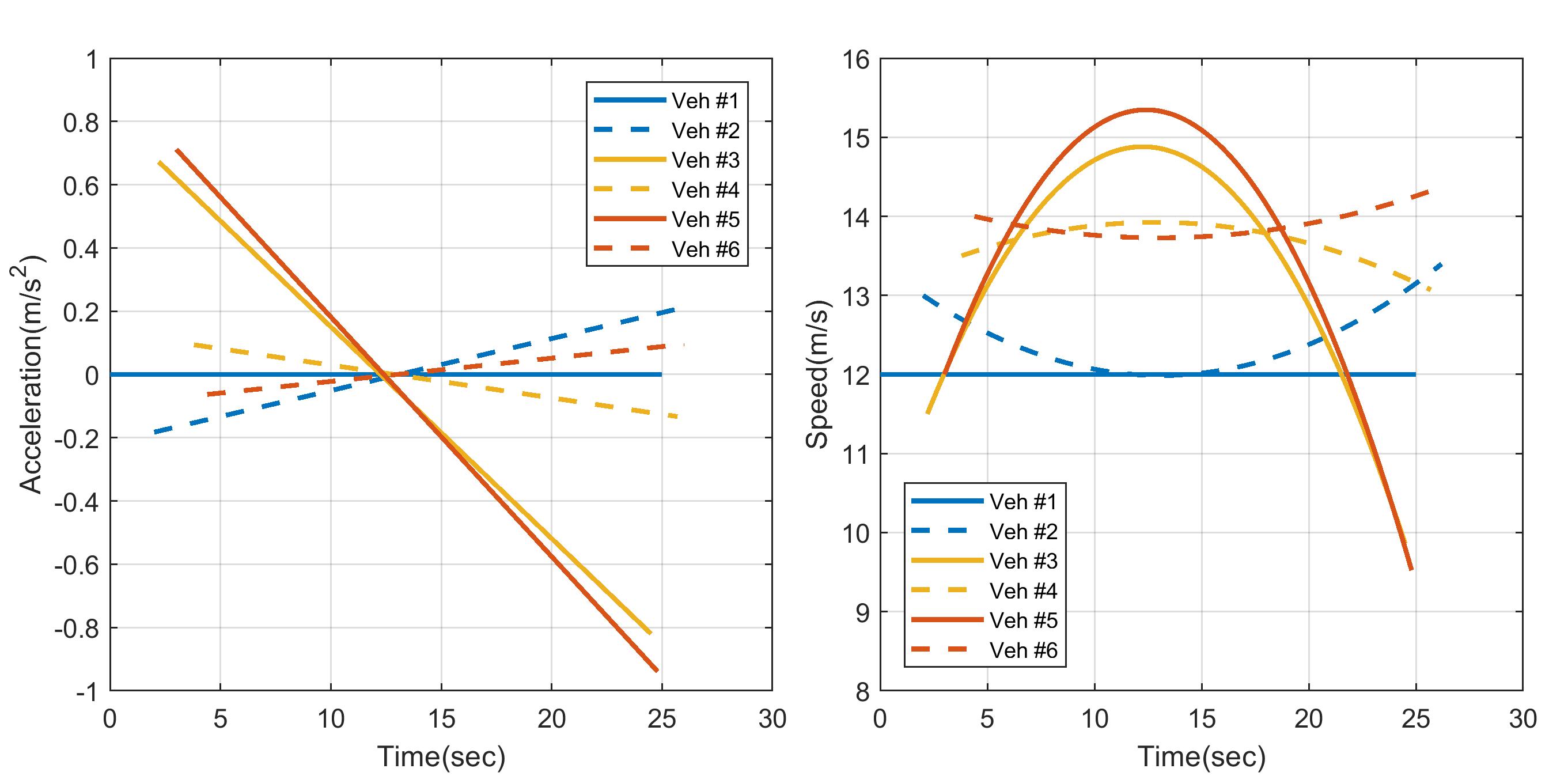}
    \caption{Speed and control profiles of vehicles.}
    \label{fig:speed}
\end{figure}
\section{Concluding Remarks}
In this paper, we formulated an upper-level optimization problem, the solution of which yields, for each CAV, the optimal sequence and lane to cross the intersection. The effectiveness of the solution was illustrated through simulation. We showed, through numerical results, that vehicles are successfully crossing an intersection without any rear-end or lateral collision. In addition, the state and control  constraints did not become active for the entire trajectory for each vehicle. 
While the potential benefits of full penetration of CAVs to alleviate traffic congestion and reduce energy have become apparent, different penetrations of CAVs can alter significantly the efficiency of the entire system. Therefore, future research should look into this direction.


\bibliographystyle{IEEEtran}
\bibliography{TCST_references}

\end{document}